\documentclass[11pt]{article}
\usepackage{amssymb,amsfonts}
\usepackage{amsmath,amsthm,amscd}
\usepackage[mathscr]{eucal}
\usepackage{mpag}
\usepackage{tabularx,longtable}
\usepackage{amsfonts}
\usepackage{amsmath,amscd}
\usepackage[russian]{babel}
\usepackage{graphicx}
\usepackage{MnSymbol}
\usepackage{amssymb,latexsym, amsmath, amscd, array, graphicx,pb-diagram,relsize}
\newtheorem{thm}{Theorem}
\newtheorem{cor}[thm]{Corollary}

\newtheorem{prop}[thm]{Proposition}

\newtheorem{defn}[thm]{Definition}
\newtheorem{rem}[thm]{Remark}

\newcommand{\Wi}{\widetilde}

\def\RR{\mathbb R}
\def\ZZ{\mathbb Z}

\def\F{\cal F}

\def\F{{\mathbf F}}

\def\Int{\mathring}
\def\F{\cal F}

\begin{document}

\title{
 Foliated Milnor conjecture}

\author{Dmitry V. Bolotov}

\address{
 47 Nauky Ave., Kharkiv, 61103, Ukraine 
 \\
\smallskip {\rm E-mail: bolotov@ilt.kharkov.ua}}

\udc{}

\date{Received May 30, 2013} 

\protect\maketitle \markright{Short title of the article for the headers}
\renewenvironment{proof}{\begin{trivlist}\item[]
{\quad\, P r o o f. }}{\hfill\rule{0.5em}{0.5em}\end{trivlist}}
\newtheorem{theorem}{\,\quad Theorem}              
\newtheorem{lemma}{\,\quad Lemma}                  
\newtheorem{definition}{\,\quad Definition}        
\newtheorem{corollary}{\,\quad Corollary}          
\newtheorem{proposition}{\,\quad Proposition}      
\newtheorem{example}{\,\quad\rm E x a m p l e }    
\newtheorem{remark}{\,\quad\rm R e m a r k }       
\renewcommand{\thesubsection}{\arabic{subsection}.}
\renewcommand{\theequation}{\arabic{equation}}
\renewcommand{\thetheorem}{\arabic{theorem}}
\renewcommand{\thelemma}{\arabic{lemma}}
\renewcommand{\thedefinition}{\thesubsection\arabic{definition}}
\renewcommand{\thecorollary}{\thesubsection\arabic{corollary}}
\renewcommand{\theproposition}{\thesubsection\arabic{proposition}}
\renewcommand{\theremark}{\thesubsection\arabic{remark}}
\renewcommand{\theexample}{\thesubsection\arabic{example}}

\begin{abstract}

We prove that a fundamental group  of codimension one nonnegative Ricci curvature $C^2$-foliation of   a  closed  Riemannian manifold is  finitely generated and almost abelian, i.e. it contains abelian subgroup of finite index.  In particular, we confirm the Milnor conjecture for manifolds   which are leaves of codimension one  nonnegative Ricci curvature foliation of closed manifold.   
\vskip2mm

{\em  Key words}: codimension one foliation,  fundamental group,  Ricci curvature,holonomy \smallskip

{\em Mathematics  Subject  Classification  2010}: 53A05.

\end{abstract}\smallskip

\begin{center}
\subsection{Introduction}
\end{center}
In 1963   Bishop proved  the following theorem  \cite{Bi}.
\newline

{\bf Bishop theorem}.\ A complete nonnegative Ricci curvature manifold has a polynomially volume growth of balls. 
\newline

In 1968 Milnor  showed    that a fundamental group of a complete nonnegative Ricci curvature manifold also has a polynomial growth (in the words metric) and  stated the following conjecture  \cite{Mi}.
\newline

{\bf Milnor conjecture.} \ The fundamental group of any complete Riemannian manifold with nonnegative Ricci curvature is finitely generated.
\newline	
	
 In this paper  we  confirm the Milnor conjecture for leaves of codimension one foliations of nonnegative Ricci curvature and prove the following theorem.
  \newline
  
  {\bf Main theorem.}\
 	Let $L$ be a  leaf of codimension one  nonnegative Ricci curvature foliation $\cal F$ of closed manifold $M$.  Then $\pi_1(L)$ is finitely generated almost abelian group. In particular, it  satisfies the   Milnor's conjecture.	
 \newline
 \begin{rem}\label{rem 1}\rm
 	The Main theorem obviously holds for compact leaves (see theorem \ref{Wi} bellow) and the problem is to show that the information comes from the a structure of nonnegative Ricci curvature foliation is enough to prove the theorem for noncompact leaves. 
 	\end{rem}
\begin{rem}\rm
	   The Main theorem can not be strengthened up to abelian groups. Indeed, consider  standard Reeb  foliation   $\cal F_R=\{L_{\alpha}\}$ on the round sphere $S^3$ (see \cite{T}). It is well known that leaves $L_{\alpha}$ in the induced metric have a nonnegative curvature. The riemannian product  $S^3$ with a  closed nonnegative Ricci manifold $M^n$ having non\-abelian funda\-mental group  give us a codimension one  nonnegative Ricci curvature foliation  ${\cal G=\{L_{\alpha}}\times M^n\}$ on $S^3 \times M^n$ with leaves having nonabelian fundamental group.
	\end{rem}
In the proof we essentially use  properties of almost without holonomy foliations and geometrical properties of complete nonnegative Ricci curvature manifolds which well studied by many famous mathematicians.   So  in 1972 J. Cheeger and D. Gromoll    have generalized  Toponogov's splitting theorem on the case complete nonnegative Ricci curvature Riemannian manifolds and obtained   the following results.
\newline

\begin{thm}[Splitting Theorem  \cite{CG}] \rm 
Let $M$ be a complete Riemannian manifold with Ricci curvature $Ric (M) \geq 0$ has a {\it stright line}  i.e. geodesic line $\gamma$ such that  $d(\gamma(u),\gamma(v)) =|u-v|$ for all $u,v\in \RR$, then $M$ is isometric to  Riemannian product  space $N\times \RR$.  
\end{thm}

\begin{thm}[\cite{CG}]\label{CG}\rm
	Let $ M^n $  be a complete manifold of nonnegative Ricci curvature, then: 
	\begin{enumerate}
	\item {$ M^n $ has at most two ends; }
		\item {$ M^n $ is isometric to the riemannian product $ N \times E ^ {k} $ of   the manifold $ N $ and Euclidian  factor $E^k$, where $ N $  does not contain  straight lines.} 

\item{
	If $ M^n $ is closed, then it's universal covering $ \Wi {M^n} $ isometric to the riemannian product $ P \times E ^ {k} $, where $ P $ is a compact and simply connected. Furthermore, the following  extension holds:
	
	\begin{equation}\label{1}\rm
	 1\to E\to \pi_1(M^n)\to \Gamma \to 1,  \end{equation}}
where $E$ is a  finite group and $\Gamma $ is a crystallographic group. 
\end{enumerate}
	\end{thm}

Part 3 of the  theorem \ref{CG} was generalized  in 2000 by B. Wilking in
the following theorem. 

\begin{thm} [\cite{Wi}]\label {Wi}\rm
Let  $M^n$ be complete nonnegative Ricci curvature manifold and $q: N\times E^l \to M^n$ be  regular isometric covering, where $N$ has a  compact isometry group (it holds in particular   when $N$ is closed),  then there exists  a  finitely sheeted covering  $N\times T^p\times E^{l-p} \to M^n$.  Moreover, this covering can be made isometric for some deformed Riemanniam metric on $M^n$.   If $M^n$ is closed, then \eqref{1} is equivalent to existence the following extension:
\begin{equation} \label{Wil}
0\to \ZZ^p\to \pi_1(M^n)\to F \to 1,
\end{equation}
where $F$ is finite group.
\end{thm}

\begin{center}
\subsection{Foliations}
\end{center}

Let us recall the notion of  foliation defined on $n$-dimensional manifold $M$. 
We say that a family ${\cal F}=\{F_{\alpha}\}$ of path-wises connected subsets (leaves) of $M$ defines a  foliation of dimension  $p$ (or codimension $q$, where $p+q=n$)   on $M$  if 
\begin{itemize}
 \item $\cal F$  is a partition of $M$, i.e.  $M = \underset{\alpha}\coprod F_{\alpha}$ 
 
 \item There is  a {\it foliated atlas}  ${\cal U} =\{ (U_\lambda,
\varphi_{\lambda})\}_{\lambda\in{ \Lambda}}$ on $M$. This means that  each connective component of the leaf in  the  foliated chart  with the coordinates $(x_1,\dots, x_p,y_1,\dots,y_q)$ looks like a plane
$$ y_{1}=const,\dots,y_{q}=const.$$
\end{itemize}
and
the transition maps 
$$g_{ij}=\varphi_i\circ \varphi_j^{-1}:\varphi_j(U_i\cap U_j) \to \varphi_i(U_i\cap U_j)$$
have the form

\begin{equation}\label{eq1}
	g_{ij}(x,y) = (\hat g_{ij}(x,y), \bar g_{ij}(y)),
\end{equation}
where  $x\in \RR^p, y \in \RR^q$.

The atlas ${\cal U} = \{ (U_\lambda,
\varphi_{\lambda})\}_{\lambda \in \Lambda}$ supposed at least $C^2$-smooth and good. The latter means the following:

\begin{enumerate}
	\item [1)]$\cal U$ - locally finite;
	\item [2)] $U_{\lambda}$ -  relatively compact in $M$ and $\varphi_\lambda(U_\lambda)= (-1,1)^n\subset \RR^n$;
	\item [3)] $\overline {U_i\cup U_j}\subset W_{ij}$,
	where $(W_{ij},\psi_{ij})$ - foliated chart which not necessary belonging to ${\cal U}$.
\end{enumerate}

Let $ \pi: (-1,1)^n\to (-1,1)^q$ be a natural projection to the
last $q$ coordinates.

The preimage  $P_{\lambda}:=\varphi_{\lambda}^{-1}(\pi^{-1}(x))$
is called  a local leaf. Denote the space of local leaves  by $Q_{\lambda}$. Clearly, $Q_{\lambda}\simeq (-1,1)^q$ and

$$U_\lambda=\bigcup_{x\in (-1,1)^q} \varphi_{\lambda}^{-1}(\pi^{-1}(x)).$$

A foliation  $\cal F$ is called oriented if such is the tangent to $\cal F$ $p$-dimensional  distribution $T^{\cal F}M\subset TM$, and $\cal F$ is called transversely oriented if such is some transversal  to $\cal F$ distribution of dimension $q=n-p$. If the manifold $M$ is supposed  riemannian, then transversely  orientability of $\cal F$ is equivalent to transversely  orientability of orthogonal    distribution   $T^{\cal F^\perp}M$.

\begin{center}
\subsection{Holonomy}
\end{center}

We recall the notion of {\it holonomy.}

Let  $l:[0,1]\to L$   be a closed path in the leaf  $L\in \cal
F$. It  was shown in \cite {T} that there exists a chain of the foliated charts ${\cal C}=\{U_0, \dots,
U_{n-1}, U_n=U_0 \}$,  which  cover of $l([0,1])$ such that:

\begin{enumerate}
	\item [a)] There exists a devision of the segment
	$[0,1]:0=t_0<t_1<\dots<t_{n}=1$, such that $l([t_{i},t_{i+1}])\subset
	U_i,\ i=0,\dots,n-1$;
	\item [b)] If the intersection $P_i\cap U_{i+1}\not = \emptyset$, then it is connective. It means that the local leaf $P_{i+1}$ is correctly defined.
\end{enumerate}

It can be shown that  a set of points $z\in U_0$ which are correctly define the chain $\cal C$  from the initial condition $z\in P_0(z)$ is an open in $U_0$. Thus there exists some neighborhood $O$ of the $P_0$ consisting of the local leaves for which the following  local diffeomorphism $\Gamma_l :V_0 \to (-1,1)^q$ of some neighborhood  of zero $V_0\subset
(-1,1)^q$ to $(-1,1)^q$ is well defined as follows:
$$ \Gamma_l(\pi\circ\varphi_0(P_0(z))=\pi\circ\varphi_0(P_{n}(z)).$$

In \cite{T} was shown that the local diffeomorphisms  $\{\Gamma_l :V_0 \to (-1,1)^q\}$    define the homomorphism $$ \Psi:\pi_1(L)\to G_q$$
of the fundamental group $\pi_1(L)$ to the group of diffeomorphism germs $G_q$ in the origin  $0\in \RR^q$. This homomorphism is defined up to inner automorphisms and it is called  a
{\it holonomy} of  $L$. It's image is called a {\it group of holonomy} of $L$ and denoted by $H(L)$.

Note that if the codimension one foliation  is transversely oriented then  one-sided holonomy $$\Psi^+:\pi_1(L)\to G^{+}_1$$ of the leaf  $L$  is well defined by analogous, where $G^{+}_1$ denotes the group of germs  of one-sided  diffeomorphisms  at $0$   defined on half-intervals $[0,\varepsilon)$.

Recall the following  important results  about the holonomy of leaves.

\begin{thm}[\cite{PTh}]\label{PTh}\rm
	Let $L$ be a leaf of codimension one $C^2$-foliation. If $H(L)$ has a polynomial growth then $H(L)$ is a torsion free abelian group.
	\end{thm}

Nishimori proved the next theorem which describes the behavior of  codimension one foliation  in the neighborhood of a compact leaf with abelian holonomy.

\begin{thm} [\cite{N}] \label{N} \rm Let ${\cal F}$ be a transversely oriented $C^{r} $ - foliation of codimension one of oriented $n$ - dimensional manifold $M$ and $F_{0} $ is a compact leaf  of ${\cal F}$. Suppose that $2\le r\le \infty $. Let  $T$ be a tubular neighborhood of  $F_{0} $ and  $U_{+} $ is a union of $F_{0} $ and a connective component $T\backslash F_{0} $.  Let us suppose that $H(F_{0}) $ is abelian.  Then one of the following   holds.
	\begin{enumerate}
		
		\item [1)] For any neighborhood $V$ of $F_{0} $ the restriction of the foliation  ${\cal F}|_{V\bigcap U_{+}} $ has a compact leaf different from $F_{0} $.
		
		\item [2)] There exists a neighborhood $V$ of  $F_{0} $ such that all leaves ${\cal F}|_{V\bigcap U_{+}} $ excepted $F_{0} $ are dense into $V\bigcap U_{+} $. In this case  $H(F_{0}) $ is free abelian of rank $\ge 2$.
		
		\item [3)] There is a neighborhood $V$ of  $F_{0} $ and connective closed codimension one oriented submanifold $N$ in $F_{0} $ with the following properties.  Denote by $F_{*} $ a compact manifold with boundary, obtained by gluing two copies  $N_{1} $ and $N_{2} $ of $N$ into $F_{0} \backslash N$ such that $\partial F_{*} =N_{1} \bigcup N_{2} $.  Let $f:[0,\varepsilon )\to [0,\delta )$ contracting $C^{r}- $ diffeomorphism such that  $f(0)=0$. Denote by  $X_{f} $ factor-manifold obtained from $F_{*} \times [0,\varepsilon )$ by identifying  $(x,t)\in N_{1} \times [0,\varepsilon )$ and $(x,f(t))\in N_{2} \times [0,\delta )$. After factorization we obtain  factor-foliation $F_{f} $ on $X_{f} $. It is claimed that for some $f$, described above, there exists $C^{r} $ - diffeomorphism $h:V\bigcap U_{+} \to X_{f} $, which maps each leaf of  ${\cal F}|_{V\bigcap U_{+}} $ on some leaf of  $F_{f} $. The foliation ${\cal F}|_{V\bigcap U_{+}} $ uniquely defines the class $[N]\in H_{n-2} (F_{0} ,\ZZ)$, and the germ at zero of the map $f$ is unique up to conjugate. In this case  $H (F_{0} )$ is infinite cyclic group.
	\end{enumerate}
\end{thm}

 	A foliation is called  {\it a foliation  without holonomy} if the holonomy of each leaf is trivial, and	one is called  a {\it foliation almost without holonomy} if the holonomy of noncompact leaves is trivial.
 	For example,  Reeb foliation $\cal F_{R}$ is almost without holonomy foliation on $S^3$ since all leaves of $\cal F_{R}$ except of a single compact leaf homeomorphic to  torus are homeomorphic to $\RR^2$ and have a trivial fundamental group.

Let us denote by  {\it block}    a compact saturated  subset $B$ of codimension one foliated $n$-dimensional manifold which is  $n$-dimensional submanifold with a  boundary. Recall that {\it a saturated set} of the foliation $\cal F$ of a manifold $M$ is called a subset of $M$ which is a union of leaves of $\cal  F$.  Clearly that $\partial B$ is a finite union of compact leaves.

 The following theorem is a reformulation of the deep results of Novikov \cite{No},  that relate to  foliations without holonomy and Imanishi \cite{Im},  that relate to    foliations almost without holonomy.
 
\begin{thm}\label{Im} \rm
	Let $L$ be a  noncompact leaf of a  codimension one almost without holonomy   foliation $\F$ of a closed n-dimensional manifold $M$. Then   one of the following   holds:
	\begin{itemize}
		\item [a)]  $\cal F$ is a foliation without holonomy. All leaves  are diffeomorphic to the typical leaf $L$ and    dense in $M$. We have the following group extension:
		\begin{equation}\label{3}
		1\to \pi_1(L)\to \pi_1(M) \to \ZZ^k\to 0,
		\end{equation}
		 where $k>0$ and $k=1$ iff  the foliation $\cal F$ is a locally trivial fibration over the circle.
		 
		 The universal covering $\Wi{M}$    has the form:
		$$\Wi{M} \cong \Wi L \times \RR.$$
		\item [b)] $L$ belongs to some  block $B$ such that all the leaves in the interior $\Int  B$  are diffeomorphic to the typical leaf $L$ and     all are or  dense in $\Int B$ ($B$ is called  a dense block in this case) or proper in $\Int  B$ ($B$ is called a proper block in this case). 
		We have  the following group extension:
		\begin{equation}\label{3}
		1\to \pi_1(L)\to \pi_1(B) \to \ZZ^k\to 0,
		\end{equation}
		 where $k>0$ and $k=1$ iff $B$ is proper and the foliation in $\Int  B$ is a locally trivial fibration over the circle. 
		 
		 The universal covering of $\Int  B$   has the form:
		\begin{equation}\label{4}
		\Wi{\Int B} \cong \Wi L \times \RR.
		\end{equation}
	\end{itemize}
\end{thm}

	Let us call the   blocks from  $b)$   {\it an  elementary blocks}.

\begin{center}
\subsection{Growth of  leaves}
\end{center}

A {\it minimal set} of the foliation  $\cal  F$ is  a closed saturated set   which does not contain different closed saturated set.

The following Plant's theorem describes  minimal sets of codimension one foliations with  leaves of subexponential growth. The growth  means the   growth of a  volume of the balls $B_x(R)\subset L_x \in \cal F$ as a function of the radius $R$.

\begin{thm} [\cite{P}] \label {P} \rm Let us suppose that a  codimension one $C^{2}$ - foliation of a compact manifold  has leaves  of  subexponential growth then  each minimal set of the foliation is or  whole manifold or a compact leaf .
\end{thm}

We say that $\cal F$ is  a  {\it nonnegative Ricci curvature foliation} if each leaf of $\cal F$ has nonnegative Ricci curvature in the induced metric.

 We obtain the following corollary.

\begin{cor}\label{cor1} \rm One of the following holds: 
		
		1) All leaves of a codimension one nonnegative Ricci curvature  $C^{2} $ - foliation are  dense.
		
		2)The closure of each leaf contains a compact leaf.
\end{cor}

A leaf $L$ of transversely oriented codimension one foliation $\cal F$ is called {\it resilient} if  there is a transversal arc $[x,y)$, $x\in L$,  and a loop $\sigma $ such that $\Gamma_ {\sigma} :[x,y)\to [x,y)$ is  the contraction to $x$ and $L\cap (x,y) \not=\emptyset$.

It turns out that the following theorem holds.

\begin{thm}[\cite{HH}] \label{res} \rm
	Let $M$ be a compact manifold and $\cal F$ be a codimension one  $C^2$-foliation of $M$. A  resilient leaf of $\cal F$  must have exponential growth. 
\end{thm}

\begin{cor}\label{cor11} \rm A codimension one $C^2$-foliation of nonnegative Ricci curvature of a compact manifold does not contain resilient leaves.
\end{cor}

\begin{center}
	\subsection{ Maps to a foliation}
\end{center}

Let us  recall the definition of  an exponential map along a foliation $\cal F$ of  riemannian manifold $M$. 

We associate with each vector $a$ tangent to a leaf $L_x$ at the point $x\in M$ the end of the geodesic of $L_x$ of length $|a|$,   starting at the initial point $x$ in the direction $a$.  Since the foliation is supposed smooth , the constructed exponential map  $\exp^{\cal F}: T^{\cal F}M\to M$ is  smooth too.
By analogous   $\exp^\perp$ denotes the orthogonal exponential map which associates with each orthogonal  vector $p$ at  $x$ the end of the orthogonal to $\cal F$ trajectory of length $|p|$  in the direction $p$ with the initial point $x$. 
Let us consider the composition of continuous maps 
\begin{equation}\label{eq10} 
F:I\times D^{n-1}(R) \overset{i}\to V(J,R)\overset {\exp^{\cal F}}\to M,
\end{equation}
 where $I=[0,1]$,  $D^{n-1}(R)$  is euclidean ball of radius  $R$  and    $i: I\times D^{n-1}(R) \to V(J,R)\subset T^{\cal F}M$ is an  homeomorphism on the set $V(J,R)=\{(x,v_x): x\in J =  i(I\times 0), v_x \in T^{\F}_xM, |v_x|\leq R \}$. The interval  $J=i(I\times 0)$ is the embedded orthogonal to $\F$  interval.     For each $x\in \Int J$ call the set $F(I,D^{n-1}(R))$ a $V(J,R)$-neighborhood of $x$.  	
\newline

\begin{rem}\label {rem7} \rm
	Since a length of $J$ and $R$ can be taken arbitrary small, 
	for any point $x$ there exists  $J$ containing $x$ and $V(J,\varepsilon)$-neighborhood W of $x$    which is contained in some foliated chart $(U_{\alpha},\phi_{\alpha})$ of a foliation. The restriction  of $\phi_{\alpha}$ on the $U'_{\alpha} \subset W$ give us foliated chard $(U'_{\alpha},\phi'_{\alpha})$ with diameters of local leaves less then $\varepsilon$, where $U'_{\alpha}$  is a neighborhood of $x$ homeomorphic to $(-1,1)^n$.  
	\end{rem}

\begin{prop}\label{pr1}  \rm Let $F:I\times D^{n-1}(R)\to M$ as in \eqref{eq10}. Then there exists $\delta>0$  such that  for each   $0\leq t<\delta$ the balls 
	$B_t(R):=\exp^{\cal F}\circ i(t\times D^{n-1}(R))$ belong to  arbitrary $\varepsilon$ - collar of the ball $B_0(R+\varepsilon)\supset
	B_0(R)$ . By $\varepsilon$ - collar of $B_0(R+\varepsilon)$ we understand the set 
	$\Delta=\exp^\perp(a), \  |a|<\varepsilon$,  where $a$ are normal to  $B_0(R+\varepsilon)$ vectors in the direction of the a half-space defined by $J=F(I\times 0)$, and $\varepsilon$ is small enough to deformation retraction $pr:\Delta \to B_0(R+\varepsilon)$ would be well defined, where $pr$ associates with the point $x\in\Delta$ the  initial point of the orthogonal  trajectory.  	
\end{prop}
\begin{proof}
	The proof follows from the fact that   $I\times D^{n-1}(R)$  is a compact for each $R\in \RR_+$ and $F$ is uniformly continuous. 
\end{proof}

Recall also the  following definition which was done by S. Adams and G. Stuck. 

\begin{defn}\label{9} \rm
	Let $\cal F$ be a foliation of riemannian  manifold $M$. 
	Let $X$ be a connected locally compact Hausdorff space. Define $C_{\cal F} (X,M)$ to be the space of continuous maps of $X$  into leaves of $\cal F$. Let us  consider $C_{\cal F}$ (X,M) as a subspace of the space $C(X,M)$	of continuous maps of $X$ into $M$, where the latter is endowed with the topology of uniform convergence on compact sets. 
\end{defn}

They have the following important result.

\begin{thm}[\cite{AS}]\label{AS} \rm
	Let $M$,  $\cal F$ and $X$ as in the definition \ref{9}.
	Then $C_{\cal F}(X,M)$  is a closed subspace of $C(X,M)$.
\end{thm}

\begin{center}
		\subsection{Nonnegative Ricci curvature foliations}
\end{center}

	\begin{prop}\label{pr2} \rm {Let $\cal F$   be a codimension one transversely orientable $C^2$-foliation of  nonnegative Ricci curvature of a closed orientable manifold $M$.  Then  $\cal F$ is a foliation almost without holonomy.}
	\end{prop}
	
	\begin{proof}
		As noted at the beginning of the article,  Milnor proved   that a  fundamental group of a compact nonnegative Ricci curvature manifold  has a polynomial growth in the words metric. Therefore by  theorem \ref{PTh}, each compact leaf has an abelian holonomy and the  theorem  \ref{N} is applied. Suppose there exists  noncompact leaf $L$ with nontrivial holonomy. Since $\cal F$ supposed transversely oriented the holonomy must be infinite. Let $\gamma $ be a closed path in  $L$ which represents  the nontrivial holonomy element of  $\pi _{1} (L)$.
		\newline
		
		\textit{Case 1.}  The one-sided holonomy $\Psi ^{+} ([\gamma ])$ is nontrivial and the holonomy map $\Gamma_{\gamma} :[0,\varepsilon )\to [0,\varepsilon ')$ does not contain fixed points in $[0,\varepsilon )$ for some  $\varepsilon $.
		\newline
		
		In this case  $\Psi ^{+} (\pm [\gamma ])$ is represented by contracting map. If $L$ is locally dense, clearly that $L$ must be a  resilient leaf that is impossible by corollary \ref{cor11}.  Otherwise another noncompact leaf $P$ is wound  on $L$ and by corollary \ref{cor1} there exists a compact leaf $K\subset \bar{L}\bigcap \bar{P}$. From  part 3 of theorem \ref{N} simply follows that the leaf  $P$ must have infinitely many ends that contradicts to  theorem \ref{CG}.
		\newline
		
		\textit{Case 2.} The  half-interval  $[0,\varepsilon )$ contains the convergent to zero sequence of  fixed points  $\{ F_{i} \} $ of the holonomy map  $\Gamma_{\gamma} :[0,\varepsilon )\to [0,\varepsilon ')$.
		\newline
		
		Since $\{ F_{i} \} \bigcap [0,F_{k} ]$ is closed in $[0,\varepsilon )$, where $F_{k} \in \{ F_{i} \} \bigcap [0,\varepsilon )$,  then either holonomy of  $L$ must be trivial or we find the half-interval $[a,\delta )\subset [0,\varepsilon )$ on which $\Gamma_{\gamma'} $ is contracting.  The leaf $L'$  corresponding to the point  $a\in [0,\varepsilon )$ has contracting holonomy on  $\gamma '\subset L'$, where $\gamma '$ is a closed path, corresponding to the fixed point $a$ of the map $\Gamma_{\gamma} $. Since the set of compact leaves is closed  (see \cite{T}) and  $M$ is normal topological space, we can choose  $\varepsilon $  small enough to propose that the leaf  $L'$ is noncompact. 
		Thus we arrive to a considered case $1$ and the proposition is proved. 
	\end{proof}

	\begin{cor}\label{Ric} \rm
		The structure of codimension one transversely oriented nonnegative Ricci curvature foliations of compact oriented manifolds is  described by theorem \ref{Im}. 
	\end{cor}

	\begin{prop}\label{str} \rm
		Let $\cal F$ be a codimension one transversely oriented $C^2$-foliation almost without holonomy (in particular foliation of nonnegative Ricci curvature) of oriented riemannian manifold $M$. Let $\{[x_i,y_i]\}\subset L_i\in \cal F$ be a sequence of shortest (in $L_i$) geodesic segments  of length $l_i\to \infty$  and  $\{z_i\in [x_i,y_i]\}$ be a sequence of points converging to some point $z\in L$, where $L$ is noncompact leaf of $\cal F$. Suppose  that the lengths of the segments  $\{[x_i,z_i]\}$ and $\{[z_i,y_i]\}$ approach  to $\infty$. Then there is a subsequence of   $\{[x_i,y_i]\}$ converging to a straight line $l\in L$, passing through $z$.  
		\end{prop}
	
	\begin{proof}   
		Starting from some $i>i_0$ we can replace the sequence of segments $[x_i,y_i]$ by the sequence of segments  $[x^l_i,y^l_i]\subset [x_i,y_i]$  of length $l$, such that  $z_i$ is the midpoint of  $ [x^l_i,y^l_i]$.  We have $\rho_{L_i}(x_i,y_i)\geq \rho_{M}(x_i,y_i)$, where $\rho_{L_i}$ and $\rho_{M}$ denote the internal metrics in the leaf $L_i$ and $M$   respectively induced by the riemannian metric on $M$. From Arzela - Ascoli theorem (see \cite [Theorem 2.5.14] {BBI}) and applying theorem \ref{AS} we can see that the the sequence $\{[x^l_i,y^l_i]\}$ converges uniformly to the rectifiable curve $r_l \subset L$ passing though $z$ with endpoints $x^l$ and $y^l$ in $L$. Show that  $r_l$ is shortest geodesic segment $[x^l,y^l]$. Indeed, if it is not true, then consider the loop $r_l\cup [x^l,y^l]$.  Consider  a smooth  transversal to $\cal F$  embedding of the square $[x^l,y^l]\times I \subset M$ such that $[x_t^l,y_t^l]:=[x^l,y^l]\times t\subset L_t\in \cal F$ and $[x_0^l,y_0^l] =[x^l,y^l]$.   The length $l_t$ of $[x_t^l,y_t^l]$  continuously depends on  $t$ (recall that $\cal F$ is $C^2$-smooth and so is induced foliation on $[x^l,y^l]\times I$) and there exist $\varepsilon$ and $\delta$ such that for  $0\leq t< \delta$ the $l_t < l-2 \varepsilon$.
		But $x^l_i \to x^l$ and  we can assume that for   $i>i_0$  $x^l_{t_i}$ and        $x^l_i$ belongs to  common local leaf of some foliated chart  $(U_{\alpha},\phi_{\alpha})$ with diameters of local leaves less then $\varepsilon$ (see remark \ref{rem7}) and therefore $\rho_{L_i}(x^l_i, x^l_{t_i})<\varepsilon$. We also assume without loss of generality that if  $i>i_0$ then $t_i<\delta$ and $\lim_{i\to \infty}t_i =0$. But    $y^l_i \to y^l$ and starting from some $i\geq i_1>i_0$ $\{y^l_i\}$ belong to a foliated chart $(U_{\beta},\phi_{\beta})$  which  contains $y^l$ and the  diameters of  local leaves of the $(U_{\beta},\phi_{\beta})$  smaller then $\varepsilon$.   If $y^l_{t_i}$ belongs to the same local leaf as $y^l_i $, we obtain the contradiction with shortest  of segments $[x^l_i,y^l_i]$: $l= \rho_{L_i}(x^l_i, y^l_{i}) \leq \rho_{L_i}(x^l_i, x^l_{t_i}) + \rho_{L_i}(x^l_{t_i},y^l_{t_i}) + \rho_{L_i}(y^l_i, y^l_{t_i}) = 2\varepsilon +l_{t_i}<l$. It means that for $i>i_1$  $y^l_i$ and $y^l_{t_i}$ belong to different local leaves of $(U_{\beta},\phi_{\beta})$ and the loop $r_l\cup [x^l,y^l]$ represents nontrivial holonomy of $L$. This is impossible since $\cal F$ is supposed almost without holonomy. This implies that  $r_l = [x^l,y^l]$.  So we have  a sequence of  shortest geodesic segments $\{[x^{l_j},y^{l_j}]\}$ passing through midpoint $z\in L$,  which are limits of sequences of shortest segments $\{[x_i^{l_j},y_i^{l_j}]\}\subset L_i$, where $l_j\to \infty$.  The sequence $\{[x^{l_j},y^{l_j}]\}$ contains a  subsequence converging to a straight line in $L$ (see for example \cite {BBI}).  
		\end{proof}

\begin{center} 
 \subsection {The main result}
\end{center} 

  \begin{defn}\rm
 Let $L$ be a   noncompact leaf of the elementary block $B$  and $K\in \partial B$ is a compact leaf. 	We call the curve  $\gamma(t)\subset L, \ t\in [0,\infty)$ {\it outgoing} (to $K$) if for each $\varepsilon >0$ there is  $t_0$, such that  $\gamma(t), t\in [t_0,\infty)$ is inside of   $\varepsilon$ - collar of $K$.  Clear that  $\bar\gamma\cap K\not = \emptyset$.
 \end{defn}

 \begin{defn}\rm
 	Let $L$ be a   noncompact leaf of the elementary block $B$ and $K\in \partial B$ is a compact leaf.
 We say that an element $[\alpha]\in \pi_1(L,x_0)$ is not peripheral  along outgoing $\gamma \subset L$ (to $K$),   if there is some 
 $\varepsilon$- collar $ U_{\varepsilon}K$ of $K$ and there is no loop $\alpha_t$ with base point $x_t = \gamma (t)$ such that $\alpha_t$  is free  homotopic to $\alpha$ and $\alpha_t$ belongs to $ U_{\varepsilon}K$.  Otherwise we say that $\alpha$ is peripheral along $\gamma$.  
 	\end{defn}
 
\begin{rem}\rm The term "peripheral"  is borrowed from the work \cite {Bu}.
	But our definition of peripheral elements of $\pi_1(L)$ along outgoing curve $\gamma$ is  in some sense foliated analogue  to  the   definition  of   fundamental group's elements having geodesic loops to infinity property along a ray $\gamma$ which was done in the work \cite{So}. 
	\end{rem}

 Let us call a {\it splitting dimension} of the  complete nonnegative Ricci curvature manifold $M$    a dimension $k$ of $E^k$ in the splitting  theorem.  
 \newline

 Now we turn to the proof of the Main theorem.

 \begin{proof} The main theorem holds for compact leaves (see remark \ref{rem 1}). Suppose $L$ is noncompact.
 	 	Suppose also that  $\cal F$ is transversely oriented and $M$ is oriented. 
  We have two cases:

	1) {\it {\cal F} does not contain compact leaves.}
	\newline
	
	In this case according to theorem \ref{Im} and corollary \ref{Ric}  all leaves are dense and diffeomorphic.  The result follows from the main theorem of \cite{AS} which claims that almost all leaves  split  as  riemannian product $S\times E^n$, where $S$ is a compact. 
\newline	

 	2) {\it {\cal F}  contains  compact leaves and  $L$ is a typical leaf of  the elementary block $B$.}
 	\newline 
 	
 	 Recall that we have the monomorphism $\pi_1(L) \to \pi_1(B)$ (see \eqref{3}).
  Fix typical leaf $L$ and consider the connective component of its universal covering $\Wi L \subset \Wi{\Int B}$. Recall that the universal covering of $\Int B$ has the form \eqref{4}. According to theorem \ref{CG} 
  
  \begin{equation} \label{5}
 	\Wi L\cong N\times E^k, \ k\geq 0,
 	\end{equation}
  	where $N$ does not contain  straight lines. Let us  show that each leaf $\Wi L_{\Wi x}$ passing though $\Wi x$ of pull back foliation $\Wi {\cal F}$ of $\Wi {\Int B}$ which is preimage of $L_{ x} \in \bar L\in \Int B$ splits with the same splitting dimension $k$ as $\Wi L$. Indeed, let $\Wi x\in \Wi L_{\Wi x}$ be an arbitrary point. We can find a sequence $\Wi x_i\in \Wi L_{\Wi x_i}$ such that $\lim_{i\to \infty}\Wi x_i = \Wi x$, where $\Wi L_{\Wi x_i}$ are preimages of $L$ passing through $\Wi x_i \in \Wi{\Int B}$ with respect to covering $p:\Wi{\Int B} \to \Int B$. Let us note that actually nontrivial is the case of dense $B$ only.  By \eqref {5} we can find a sequence of $k$ orthogonal straight lines passing through $\Wi x_i$, which converge to $k$ orthogonal straight lines passing through $\Wi x$. This follows from some modification of proposition \ref{str}.
 	 It follows  that $\Wi L_{\Wi x}\cong N_x\times E^l$, $l\geq k$. Changing $L$ and $L_{ x}$ we obtain $l=k$. It is clear that splitting  directions $\RR^k$ form a subdistribution in $T^{\cal F}M$. Let us call it a splitting distribution and orthogonal to it in $T^{\cal F}M$ we  call a horizontal distribution.      
 	
 	Let us suppose that $N$ is not compact in the splitting $\Wi L\cong N\times E^k$. Then there is an  horizontal ray $\Wi \gamma  \in N\times \vec 0$. Show that $p(\Wi \gamma) =: \gamma$ is an outgoing curve. Suppose that $\gamma$ is not outgoing. Then we can find a  sequence  of points $\gamma (t_i), \ t_i \to \infty$, which converge to point $x\in \Int B$. It means that we can find a sequence of points $\Wi x_i:= \Wi \gamma_i(t_i)$ which converge to the point $\Wi x$ such that $p(\Wi x) = x$. Here $\Wi \gamma _i$ are corresponding preimages of $\gamma$ with respect to covering map $p:\Wi{\Int B} \to \Int B$. Since $t_i\to \infty$  we can find  a sequence of shortest horizontal  segments  $[a_i,b_i]\in \Wi \gamma_i$ through $\Wi x_i$ such that  lengths $[a_i,\Wi x_i]$ and $[b_i,\Wi x_i]$ limits to $ \infty$ and thus  it contains a  subsequence converging to the horisontal straight line through $\Wi x$ by proposition \ref{str}. It is a contradiction since $N_{\Wi x}$ in the splitting $\Wi L_{\Wi x} \cong N_{\Wi x} \times E^k$ does not contain straight lines.    Thus $\gamma$ is outgoing.

 	Consider two possibilities.
 	\newline
 	
 	a) {Each element $[\alpha] \in \pi_1(L)$ is  peripheral along  outgoing $\gamma$ to  $K\in \partial B$. }
 		\newline
 		
 		Suppose $\gamma (t) \subset U_{\varepsilon}K$ for $t\in [t_0,\infty)$. In this case we can change  each loop $\alpha$ in $L$ with the base point in $\gamma (t_0)$ by the homotopic loop $\alpha'\subset U_{\varepsilon}K$ with the same base point.   Since $K$ is a deformation retract of $U_{\varepsilon}K$ we have that $i_*(\pi_1(L))$ is isomorphic to  subgroup of $ i_*(\pi_1(K))$, where $i_*$ denotes  in both cases homomorphisms induced by inclusions  $i:L\to B$ and $i:K\to B$.    But a subgroup and  an image of a finitely generated almost  abelian group is finitely generated almost  abelian, thus the result follows from  \eqref{Wil} and theorem \ref{CG}. 
 		\newline
 		
 		b) {There is a non peripheral element $[\alpha]\in \pi_1(L)$ along  $\gamma$.}
 		\newline
 		
  From the definition follows  that there is a countable family of the free homotopic to the loop $\alpha$,  minimal length  geodesic loops $\{ \alpha_i\}$ with the base points $x_i = \gamma (t_i), \ t_i \to \infty$  and the sequence $\{z_i\in\alpha_i\}$ such that $z_i\in  B\setminus U_{\varepsilon}K$ for some $\varepsilon$-collar $U_{\varepsilon}K$ of $K\in \partial B$. Let us  suppose that $\{z_i\}$ converge to $z\in B\setminus U_{\varepsilon}K$.  Since  $\gamma$ is outgoing we can choose subsequence $\{x_j\}$ in $\{x_i\}$ converging to $x\in K$ and  we can see that $\rho(x_j,z_j) \to \infty$. Indeed, by proposition  \ref{pr1} for arbitrary fixed $R> diam \ K$ there exist $\delta >0$  such that  for each   $0\leq t<\delta$ the balls $B_t(R):=\exp^{\cal F}\circ i(t\times D^{n-1}(R))$ belong to   $\omega$ - collar $U_{\omega}K$ ($\omega < \varepsilon$), where $i(0,0)=x$.
 Since $x_i\to x$, there exists $i_0$ such that  $x_i \in  B_t(\sigma) \subset B_t(R)$ for $i>i_0$  and fixed   small enough  $\sigma>0$ (see remark \ref{rem7}) and $\rho_{L_i}(x_i, z_i)\geq R - \sigma$. Since $R$ is taken arbitrary  the result follows.
  
  But it means that we have a sequence of the  shortest geodesic segments  $[\Wi x_j,\Wi y_j] \subset  \Wi L\subset \Wi{\Int B}$ of length $l_j\to \infty$  with the ends $\Wi x_j \in \Wi \gamma \subset \Wi L$ and  $ \Wi y_j \in \Wi \gamma_{\alpha} := \Theta ([\alpha]) (\Wi \gamma )\subset \Wi L$ such that  $p(\Wi x_j) = p(\Wi y_j) = x_j$, where $\Theta$ defines an isometric action of $\pi_1(B)$ on $\Wi {\Int B}$.   
 		    Let $\Wi z_j \in [\Wi x_j,\Wi y_j]$ be such that $p(\Wi z_j)=z_j$.
 		    Since $\rho(x_j,z_j) \to \infty$ we have $\rho(\Wi x_j, \Wi z_j) \to \infty$ and $\rho(\Wi y_j, \Wi z_j) \to \infty$.
 		     Choose $g_j\in\pi_1(B)$   such that  $\lim_{j\to \infty} \Theta(g_j)\Wi z_j \to \Wi z \in \Wi {\Int B}$ and $p(\Wi z) = z$. 
 		    Pay attention to that $[\Wi x_j,\Wi y_j]$ is not need to be horizontal, since $\Wi \gamma_{\alpha}$ is not need to belong to $N\times \vec 0$, but it belongs to $N\times \vec a$ for some $\vec a  \in \RR^k$. Change the ray $\Wi\gamma$  to the ray $\Wi \gamma ':= \Wi \gamma + \vec{a}$
 		    and $\Wi x_j$ to $\Wi x'_j:=\Wi x_j + \vec a$. Clearly that shortest geodesic segments 
 		    $[\Wi x'_j,\Wi y_j]$ are the projections of $[\Wi x_j,\Wi y_j]$ along $\vec a$. Let $z'_j \in [\Wi x'_j,\Wi y_j]$ is images of $z_j$ with respect to the projections. We can see that $[x_j,x_j',z_j]$ is right triangle.
 		    Since $\rho(x_j,x'_j) =|a|$ we have that $lim_{j\to \infty}\rho(x'_j,y_j) = \infty$ and $\rho(z_j,z'_j)< |a|$. This immediately implies from proposition \ref{str} that a sequence of shortest horizontal geodesic segments $\Theta(g_j)[\Wi x'_j,\Wi y_j]$ has a subsequence converging to the horizontal straight line that  is impossible. 
 		    \newline
 		    
 		    If  $N$ in the splitting $\Wi L \cong N \times \RR^k$ is compact then the result follows from theorem \ref {Wi}. 
 		    
 		If $M$ is not orientable and (or) $\cal {F}$  is not transversely orientable  we can go to the finitely sheeted oriented isometric covering $p:\bar M \to M$   such that pull-back  foliation $\bar{\cal F}$ is transversely oriented and for which the result was proved.  Each leaf $L \in \cal F$ is finitely   covered by some leaf $\bar L \in \bar{\cal F}$. And the  result in general case follows from the fact that  $p_*:\pi_1(\bar L)\to  \pi_1 (L)$ is a  monomorphism  and $\pi_1(\bar L )$ is isomorphic to some subgroup of finite index in $\pi_1 (L)$.
 	\end{proof}

 \begin{cor} \rm The leaves of a codimension one nonnegative Ricci curvature foliation of a closed riemannian manifold satisfy the   Milnor's conjecture.
 	\end{cor}
 
 The author is grateful  to the  referee for useful comments.

\renewcommand{\refname}{\centering

\end{document}